\documentclass[12pt]{article}
\usepackage[latin2]{inputenc}
\usepackage{authblk}

\usepackage{authblk}

\usepackage{graphicx}
\usepackage{amssymb}
\usepackage{amsmath}
\usepackage{bbm}

\newcommand{\R}{{\ensuremath{\mathbb{R}}}}

\newcommand{\Z}{{\ensuremath{\mathbb{Z}}}}
\newcommand{\1}{\mathbb{I}}
\newcommand{\X}{\boldsymbol{X}}
\newcommand{\Y}{\boldsymbol{Y}}
\newcommand{\ZZ}{\boldsymbol{Z}}
\newcommand{\Q}{\boldsymbol{Q}}

\newcommand{\x}{\boldsymbol{x}}
\newcommand{\tB}{\boldsymbol{t}}
\newcommand{\sB}{\boldsymbol{s}}
\newcommand{\bt}{\boldsymbol{\Theta}}
\newcommand{\be}{\boldsymbol{\eta}}
\renewcommand{\P}{\ensuremath{\mathbb{P}}}
\renewcommand{\dj}{d\kern-0.4em\char"16\kern-0.1em}
 \newcommand{\fidi}{\xrightarrow{\text{fidi}}}
 \newcommand{\dto}{\xrightarrow{d}}
\newcommand{\E}{\mathbb{O}}
\newcommand{\EE}{\mathbb{E}}

\newcommand{\toi}{\to\infty}

 \newcommand{\dsum}{\displaystyle\sum}
  \renewcommand{\le}{\leqslant}

     \newcommand{\eqd}{\stackrel{d}{=}}
     \newcommand{\asto}{\xrightarrow{as}}
    
    \newcommand{\vto}{\xrightarrow{v}}   
 \newcommand{\pto}{\xrightarrow{P}}
 
 \newcommand{\vep}{\varepsilon}

 \newtheorem{Thm}{Theorem}[section]
  
\newtheorem{Prop}[Thm]{Proposition}

\newtheorem{Rem}{Remark}
\newenvironment{proof}{\noindent \textbf{Proof.}\, \, }{\null\hfill$\blacksquare$\hskip 2mm\vskip2mm}
\numberwithin{equation}{section}
\numberwithin{equation}{section}
\title{Complete convergence theorem for stationary heavy tailed sequences}
\author[*]{Bojan Basrak}
\author[**]{Azra Tafro}
\affil[*]{Department of Mathematics, University of Zagreb, Croatia\\ bbasrak@math.hr}
\affil[**]{Department of Mathematics, University of Zagreb, Croatia\\ atafro@math.hr}

\begin{document}
\maketitle
\begin{abstract}
For a class of stationary regularly varying and weakly dependent multivariate time series $(\X_n)$,
we prove the so-called complete convergence result for the space--time point processes
of the form  $N_n  = \sum_{i=1}^n \delta_{(i/n, \X_i/a_n)}.$ 
As an application of our main theorem, we give a simple proof of the invariance principle
 for the corresponding partial maximum process.

 
\end{abstract}

\section{Introduction}
Point processes theory is widely recognized as a useful and  elegant tool for the extremal analysis of stochastic processes.  This approach is  splendidly illustrated by Resnick \cite{Res08} or Leadbetter and Rootz\'en \cite{LeRo88}.  It is well known for instance that for an iid sequence of random variables $(X_n)$, 
regular variation of the marginal distribution is equivalent to the so called complete convergence result,
that is to the convergence of 
point processes of the form
\begin{equation} \label{e11}
N_n  = \sum_{i=1}^n \delta_{(i/n, X_i/a_n)},
\end{equation} 
towards a suitable Poisson point process.
Such a statement then  yields nearly all relevant asymptotic distributional properties about the sequence 
$(X_n)$, cf. \cite{LeRo88}.  There exists a rich literature on extensions of this result to dependent stationary sequences in both univariate and multivariate cases, 
for an illustration consider for instance
Davis and Resnick~\cite{DR85},
Davis and Hsing~\cite{DH},
Davis and Mikosch~\cite{DM98},
Hsing and Leadbetter~\cite{HL98} or
Basrak et al. ~\cite{BKS}. 
One of the earliest and key results in this area was obtained by Mori \cite{Mo77} who
showed that   all possible limits for the point processes $N_n$ have the form of
a Poisson cluster process provided that the random variables $X_n'$s are strongly mixing with a strictly positive extremal index. In such a case the limit can be written
as
\[ 
 N=\sum_{i}\sum_{j}\delta_{(T_i, P_i Q_{ij})},
\]
where $\sum_{i}\delta_{(T_i,P_i)}$ is a suitable Poisson process and 
$(\sum_{j}\delta_{Q_{ij}})_i$ represents an independent iid sequence of point processes
on ${\R}\setminus \{0\}$ with the property that  $\max_{j}{|Q_{ij}|}=1$ for all $i$.
However,  the relationship between 
the sequence $(X_n)$ and 
the 
distribution of the limiting process  $N$ or its components is up to now only partly understood.
Due to the dependence, for a general sequence $(X_n)$, there is no 
 complete convergence result which determines the shape of the limiting process
 $N$  in the form suggested by Mori.
It is our main goal here to give such a result for a rather wide
 class of weakly dependent  regularly varying processes, even in the multivariate setting.

%

 Recall that a $d$-dimensional random vector $\X$ is regularly varying with index $\alpha > 0$  if there exists a random vector $\boldsymbol{\Theta} \in \mathbb{S}^{d-1}$ such that 
\begin{equation}\label{def:RV_d}
\frac{1}{\P(\|\X\| > x)}\P(\|\X\| > ux, \X / \|\X\| \in \cdot )\Rightarrow u^{-\alpha} \P(\bt \in \cdot),
\end{equation}
for every $u>0$ as $x \toi$, where $\Rightarrow$ denotes the weak convergence of measures. Note that the definition does not depend on the choice of the norm, i.e. if \eqref{def:RV_d} holds for some norm in $\R^d$, it holds for all norms, with different
distributions of $\bt$ clearly. 
A $d$--dimensional time series $(\X_n)_{n\in \Z}$ is regularly varying if all of the finite-dimensional vectors $(\X_k,\dots,\X_l),\  k,l \in \Z$ are regularly varying, see Davis and Hsing~\cite{DH} for instance. We will consider a strictly stationary regularly varying process $(\X_n)_{n\in \Z}$. This means in particular, that there exists a sequence $(a_n)$, $a_n\toi$ such that
\begin{equation}\label{e:a_n}
n\P(\|\X_0\| > a_n x)\to x^{-\alpha}\quad \textrm{ for all } x>0.
\end{equation}
According to Basrak and Segers~\cite{BS}, the regular variation of the stationary sequence $(\X_n)$ is equivalent to the existence of the \textit{tail process} $(\Y_n)_{n \in \Z}$ which satisfies $\P(\|\Y_0\| > y)=y^{-\alpha}$ for $y \geq 1$ and, as $x\to \infty$,
\begin{equation}\label{e:tailprocess}
 \left( (x^{-1}\, \X_n)_{n \in \Z} \, \big| \, |\X_0| > x \right)
  \fidi (\Y_n)_{n \in \Z},
\end{equation}
where $\fidi$ denotes convergence of finite-dimensional distributions. Moreover,  the so-called
\textit{spectral tail process}   $(\bt_n)_{n \in \Z}$ defined 
as a sequence $\bt_n = \Y_n/\|\Y_0\|\,,$  ${n \in \Z},$   turns out to be independent of $\|\Y_0\|$ and satisfies
\begin{equation}\label{e:theta}
 \left( (\|\X_0\|^{-1}\, \X_n)_{n \in \Z} \, \big| \, \|\X_0\| > x \right)
  \fidi (\bt_n)_{n \in \Z}.
\end{equation}
 as $x\toi$.
In the sequel, we assume that there exists a sequence $(r_n)$, where $r_n\toi$ and $n/r_n \toi$, such that $(\X_n)$ and $(r_n)$ satisfy the following two conditions.
Denote first
 $\E=\overline{\R}^d\setminus \{\boldsymbol{0}\}=[-\infty, \infty]^d\setminus \{\boldsymbol{0}\}$.\\
\smallskip

\noindent
{\bf Main assumptions}\\[2mm]
\noindent ($\mathcal{A}'$):
For every $f \in C_{K}^{+}([0,1] \times
\E)$, denoting $k_{n} = \lfloor n / r_{n} \rfloor$, as $n \to
\infty$,
\begin{equation}\label{e:A'}
 \EE \biggl[ \exp \biggl\{ - \sum_{i=1}^{n} f \biggl(\frac{i}{n}, a_n^{-1}{\X_{i}}
 \biggr) \biggr\} \biggr]
 - \prod_{k=1}^{k_{n}} \EE \biggl[ \exp \biggl\{ - \sum_{i=1}^{r_{n}} f \biggl(\frac{kr_{n}}{n}, a_n^{-1}{\X_{i}} \biggr) \biggr\} \biggr] \to 0.
\end{equation}
\noindent ($\mathcal{AC}$):
For every $u > 0$,
\begin{equation}
\label{e:AC}
  \lim_{m \to \infty} \limsup_{n \to \infty}
  \P \biggl( \max_{m \le |i| \le r_{n}} \|{\X_{i}}\| > a_n u\,\bigg|\,\|{\X_{0}}\|>a_n u \biggr) = 0.
\end{equation}

For thorough discussion of conditions ($\mathcal{A}'$) and
($\mathcal{AC}$) we refer to \cite{BKS} or Bartkiewicz et al \cite{BJM}. We observe here
only that  $m$-dependent multivariate sequences clearly satisfy both assumptions. More general strongly mixing sequences always satisfy condition $\mathcal{A}'$, which  roughly speaking,  allows one to break  the sequence  into increasing and asymptotically  independent blocks. The condition $\mathcal{AC}$ on the other hand  restricts the clustering of extremes in the sequence $(\|\X_t\|)$ over time. 
In addition to $m$-dependent sequences, it is satisfied for many other frequently used heavy tailed models,
including for instance stochastic volatility or ARCH and GARCH processes,
see \cite{BJM} and \cite{BKS}.

The main object of interest in the sequel  is the multivariate version of the  point process
in \eqref{e11}
$$
N_n=\sum_{i=1}^n \delta_{(i/n, \X_i/a_n)}
$$
on  the state space $[0,1]\times \E$.
 Throughout, the space of point measures on a given space, $\mathbb{S}$ say, is denoted by $M_p(\mathbb{S})$ and 
  endowed by the vague topology, see Resnick~\cite{Res87}.

Under very similar assumptions as above, Davis and Hsing in \cite{DH} analyzed the projection of the process $N_n$ on the spatial coordinate in the univariate case. They showed that the point processes 
$N^*_n =\sum_{i=1}^n \delta_{X_i/a_n}$, $n \in \Z$,
on the space $[-\infty,0)\cup (0,\infty]$, converge in distribution to a point process
which has a Poisson cluster structure which is only implicitly described. 
They further apply this result to determine the limiting distribution of
the partial sums in the sequence $(X_n)$.
These results were extended to the multivariate setting by Davis and Mikosch~\cite{DM98}.
 Clearly, unlike $N_n$,
the point processes $N^*_n$ contains no information about time clustering of extremes in the sequence $(X_n)$.

Point processes $N_n$ were already studied by Basrak et al. in \cite{BKS} in the univariate case, but the proof therein carries over to the multivariate case as observed in
Basrak and Krizmani\'c \cite{BK15} with some straightforward adjustments, to show
\begin{equation} \label{e:Nu}
N_n \Big|_{[0,1]\times  (\overline{\R}\setminus [-u,u])^d  } 
\Rightarrow N^{(u)} \Big|_{[0,1]\times (\overline{\R}\setminus [-u,u])^d } \,, 
\end{equation}
as $n\toi$,  for any threshold $u>0$, with the limit $N^{(u)}$ unfortunately
depending on that threshold. 
This  asymptotic results can be used to deduce  functional limit theorems for partial sums or partial maxima in the time series $(X_t)$,
 see Theorem 3.4 in \cite{BKS} or Proposition~\ref{Thm:ExProc} below.
Our goal here is to avoid  restriction to various domains in \eqref{e:Nu}, to find the correspondence of this result
with the earlier results of Davis and Hsing~\cite{DH} 
and Davis and Mikosch~\cite{DM98}.
Our main theorem below essentially unifies theorems on the limiting behavior of point processes
given in  \cite{DH,DM98,BKS}. Moreover, it reconciles their apparently very different statements in the framework suggested by 
Mori's result \cite{Mo77}.

 In the following section we show an interesting preliminary result  about the structure of extreme clusters in the process $(\X_n)$. Then, in Section~\ref{Sec:main}, we prove a general theorem
 about point process convergence for regularly varying time series. 
 In Section~\ref{Sec:etrem}, we apply this theorem to show the invariance principle
 for the so-called maximal process in the space $D[0,1]$. A corresponding theorem for iid sequences is well known and can be found in Resnick~\cite{Res87}. For nonnegative stationary regularly varying sequences, it was recently proved by Krizmani\'c~\cite{DK14}. Our version of this result
 includes other  regularly varying sequences, and since it relies on our main theorem,
 the proof is straightforward and relatively  simple. We also exhibit a simple technique to avoid problems at
 the left tail which are typically a source of frustration in this sort of limiting theorems.

\section{Preliminaries}

It was shown in \cite{BS} that under the main assumptions above,
there exists
\begin{equation}\label{theta}
\theta=
\P \left( \sup_{i\leq -1}\|\Y_i\| \leq 1\right)=
\lim_{r\to \infty}\lim_{x \to \infty}\P(M_r \leq x \big| \|\X_{0}\| > x) >0,
\end{equation}
where $M_{r}=\max\{\|\X_k\|: k=i,\dots,r\}$. 
The number $\theta>0$ 
  represents
the extremal index of the random nonnegative sequence $(\|\X_t\|)$. 

Following \cite{BS}, we define an auxiliary sequence $(\ZZ_j)_{j\in \Z}$ as a sequence of random variables
distributed as $(\Y_j)_{j\in \Z}$ conditionally on the event $\{\sup_{i\leq -1}\|\Y_i\| \leq 1\}$. 
More precisely,
\begin{equation}\label{e:z}
\mathcal{L}\left( \sum_{j\in \mathbb{Z}} \delta_{\ZZ_j}\right) = \mathcal{L}\left(\sum_{i\in \mathbb{Z}} \delta_{\Y_i}\, \Big|\, \sup_{i \leq -1} \|\Y_i\| \leq 1 \right)\,,
\end{equation}
where $\mathcal{L}(N)$ denotes the distribution of the point process $N$.
It was shown in Theorem 4.3. in \cite{BS} that
\begin{equation}
\label{E:clusterprocess}
    \mathcal{L} \biggl( \sum_{i=1}^{r_n} \delta_{(a_n u)^{-1} \X_i} \, \bigg| \, M_{r_n} >a_n u \biggr)
    \Rightarrow \sum_{j\in \Z}\delta_{\ZZ_j}.
\end{equation}
Expressing the weak convergence of point processes using Laplace functionals,  \eqref{E:clusterprocess} is equivalent to 
\begin{equation}\label{E:clusterprocesspom}
\EE \left( e^{-\sum_{i=1}^{r_n} f((a_n u)^{-1}\X_i) } \, \big|\,   M_{r_n} >a_n u \right) \to \EE\left[e^{-\sum_{j\in \mathbb{Z}}f(\Y_j)}\,\big| \, \sup_{i \le -1} \|\Y_i\|\leq 1 \right]
\end{equation}
for all $u \in (0,\infty)$ and $f \in C_{K}^{+}(\E)$.

Observe that by $\mathcal{AC}$,  $\|\Y_n\|\to 0$ with probability 1 as $|n|\toi$, see Proposition 4.2
in  \cite{BS}. Therefore the same holds for $\ZZ_n$'s in \eqref{e:z}.
In particular, the random variable
\begin{equation*}
L_Z=\sup_{j\in \Z} \|\ZZ_j\|
\end{equation*} 
is a.s. finite, and clearly not smaller than 1. To determine the distribution of $L_Z$, 
observe that from Proposition 4.2
in  \cite{BS} it follows that 
\begin{equation}\label{e:supZ}
 k_n \P ( M_{r_n} > a_n u) \to \theta u^{-\alpha}.
\end{equation}
Now  for $v \geq 1$, we have by \eqref{E:clusterprocess}
\begin{eqnarray*}
\lefteqn{
\P(L_Z>v)=1-\P\left(\sum_j \delta_{\|\ZZ_j\|}((v,\infty))=0 \right)} \\
&=&1-\lim_{n\toi} \P\left(\sum_{i=1}^{r_n}\delta_{(a_n u)^{-1}\|\X_i\|}((v,\infty))=0\,\vert\, M_{r_n} >a_n u  \right)\\
&=&\lim_{n\toi} \P(M_{r_n}>a_n uv\,\vert\,M_{r_n} >a_n u )=\lim_{n\toi}\frac{k_n \P(M_{r_n} >a_n uv)}{k_n \P(M_{r_n} >a_n u)}.
\end{eqnarray*}
Therefore,
\begin{equation}\label{e:L_Z}
\P(L_Z>v)= v^{-\alpha}.
\end{equation}
For $(\ZZ_j)$ as in \eqref{e:z}, we define a new sequence  $(\Q_j)_{j\in \Z}$ by
\[
 \Q_j=\ZZ_j / L_Z\,,\qquad j\in \Z\,.
\]
 We will show the independence between the point process  $\sum_j \delta_{\Q_j}$ and  the random variable $L_Z$, which might not be entirely surprising in the view of the 
 independence between $\|\Y_0\|$ and the spectral tail process in \eqref{e:theta}. However, this result
 seems to be new.

\begin{Prop}\label{P:main_pomocni} Assume that a regularly varying stationary sequence $(\X_n)$ 
with corresponding sequence $(r_n)$ satisfies conditions $\mathcal{A}'$ and $\mathcal{AC}$. Then
\begin{equation}\label{e:main_pomocni}
\left(\sum_{i=1}^{r_n} \delta_{M_{r_n}^{-1}\\X_i}, \frac{M_{r_n}}{a_n u}\: \Big| \: M_{r_n} > a_n u\right) \Rightarrow \left(\sum_{j}\delta_{\Q_j},  L_Z\right).
\end{equation}
Moreover, $L_Z$ and  $\sum_j \delta_{\Q_j}$ on the right hand side  are independent.
\end{Prop}
\begin{proof}
For $m= \sum_i \delta_{\x_i}\ \in M_p(\E)$ denote by $x_{m}$ the largest norm of any  point in $m$, i.e. $x_{m}=\sup_i \|\x_i\|$  and define the mapping $\phi$ by
$$\phi: m \mapsto(m, x_{m})\,.$$
As observed in \cite{DH} such a mapping is continuous. Hence,
applying $\phi$ to \eqref{E:clusterprocess}, we obtain
\begin{equation}\label{e:main_pomocni_1}
\left(\sum_{i=1}^{r_n} \delta_{(a_n u)^{-1}\X_i}, \frac{M_{r_n}}{a_n u}\: \Big| \: M_{r_n} > a_n u\right) \Rightarrow \left(\sum_{j}\delta_{\ZZ_j},  L_Z\right).
\end{equation}
Consider  for $\nu \in M_p$ and $b\in (0,\infty)$, the mapping
$$\psi: (\nu,b) \mapsto \nu_b \in M_p,$$
where $\nu_b ( \cdot ) = \nu(b^{-1}\cdot)$.
Mapping $\psi$ is again continuous by Proposition 3.18 in Resnick \cite{Res07} for instance.
 Hence, \eqref{e:main_pomocni_1} implies \eqref{e:main_pomocni}  by the continuous mapping theorem.

To show the  independence between $L_Z$ and $\sum_j \delta_{\Q_j}$, it suffices to show
\begin{equation}\label{e:nez1}
\EE \left[\exp  \left(-\sum_j f_1({\Q_j})\right) \mathbbm{1}_{\{L_Z >v\}} \right]=\EE \left[\exp  \left(- \sum_j f_1({\Q_j})\right)\right]  P({L_Z >v })\,,
\end{equation}
for an arbitrary function $f_1\in C_K^{+} (\E)$ and $v \geq 1$.
By \eqref{e:main_pomocni},  the left-hand side of \eqref{e:nez1} is the limit of 
 $$\EE \left[\exp \left(- \sum_{i=1}^{r_n}f_1({\X_i/M_{r_n}}) \right)\,\mathbbm{1}_{\{(ua_n)^{-1}M_{r_n}>v\}}\bigg\vert M_{r_n}>a_n\right],$$
 which  further equals
  $$\EE  \left[\exp \left(- \sum_{i=1}^{r_n}f_1({\X_i/M_{r_n}}) \right) \, \bigg\vert M_{r_n}>a_n v\right]\frac{\P(M_{r_n}>a_n  v)}{\P(M_{r_n}>a_n )}.$$
By \eqref{e:main_pomocni} and the continuous mapping theorem,
the first term above tends to $\EE\left[\exp\left(-\sum_j f_1({\Q_j})\right)\right]$ as $n\toi$. By \eqref{e:supZ}, the second term tends to $\P(L_Z >v)
=v^{-\alpha}$. which implies  \eqref{e:nez1}.
\end{proof}

\section{Main theorem}\label{Sec:main}
\begin{Thm}\label{main}
Let $(\X_n)_{n\in \Z}$ be a stationary series of jointly regularly varying random vectors with index $\alpha$, satisfying conditions $\mathcal{A}'$ and $\mathcal{AC}$ for a certain sequence $(r_n)$. Then
\begin{equation} \label{main_2}
N_n=\sum_{i=1}^n \delta_{(i/n, \X_i/a_n)} \Rightarrow N=\sum_{i}\sum_{j}\delta_{(T_i, P_i\be_{ij})},
\end{equation}
where 
\begin{itemize}
\item[i)] $\sum_{i}\delta_{(T_i,P_i)}$ is a Poisson process on $[0,1]\times(0,\infty)$ with intensity measure $Leb \times \nu$ where $\nu(dy)= \theta \alpha y^{-\alpha-1}dy$ for $y>0$.
\item[ii)] $(\sum_{j}\delta_{\be_{ij}})_{i}$ is an i.i.d. sequence of point processes in $\E$ independent of $\sum_{i}\delta_{(T_i, P_i)}$ and with common distribution equal to the distribution of $\sum_{j}\delta_{\Q_j}$.
\end{itemize}
\end{Thm}
\begin{proof}
To show \eqref{main_2}, it is sufficient to  show that 
$$\EE e^{-N_n(f)}\to \EE e^{-N(f)},$$
for an arbitrary $f\in C_{K}^{+}([0,1]\times \E)$. 
Observe that for any such function $f$ there is a
constant $u>0$, such that the support of $f$ is a subset of $[0,1]\times \E_u $. 
By  \eqref{e:Nu}, i.e. by the multivariate expansion of Theorem 2.3 in \cite{BKS}, we know that
\begin{equation} \label{e:Nuf}
\EE e^{-N_n(f)}\to \exp \left[-\int_0^1\left(1-\EE e^{-\sum_j f(t,u\ZZ_j)}\right) \theta u^{-\alpha} dt\right]
\end{equation}
Consequently, it suffices to show that the right hand side above corresponds to
$\EE e^{-N(f)}$ for any such function $f$ and corresponding $u>0$.

As in Theorem 2.3 in \cite{BKS}, one can write
\begin{eqnarray*}
\EE e^{-N(f)}&=&\EE \exp\{-\sum_i \sum_j f(T_i, P_i \be_{ij})\}\\
&=&\EE \left[\EE\left(\prod_i \exp\{- \sum_j f(T_i, P_i \be_{ij})\}\big\vert ((T_k, P_k))_k\right)\right]\\
&=&\EE \left[\prod_i \EE\left(\exp\{- \sum_j f(T_i, P_i \be_{ij})\}\big\vert ((T_k, P_k))_k\right)\right],
\end{eqnarray*}
where the last equation follows from Lemma 3.10 in \cite{Res87}, since $\sum_{i}\delta_{(T_i, P_i)}$ and   
$(\sum_{j}\delta_{\be_{ij}})_{i}$ on the right-hand side of \eqref{main_2} are independent.
Define $h(t,v)=\EE \exp\{-\sum_j f(t,v\Q_j)\}$. We have
\begin{equation*}
\EE e^{-N(f)}= \EE \exp \left( \sum_i \ln h(T_i, P_i) \right).
\end{equation*}
The right-hand side is the Laplace functional of $\sum_{i}\delta_{(T_i,P_i)}$ evaluated at $-\ln h$. 
 Since $\sum_{i}\delta_{(T_i,P_i)}\sim $PRM$(Leb \times \nu)$ on $[0,1]\times(0,\infty)$, 
 the Laplace functional of this process 
 has a very special form, see Resnick~\cite{Res87}.
\begin{alignat}{2}
\EE e^{-N(f)}&=\exp \left[-\int_{[0,1]\times (0,\infty)} \left(1-h(t,v)\right) (Leb \times \nu)(dt, dv)\right]\notag\\
&=\exp \left[-\int_0^1\int_0^{\infty}\left(1-\EE e^{-\sum_j f(t,v\Q_j)}\right) \theta \alpha v^{-\alpha-1}dv dt\right].\label{e:main_1}
\end{alignat}
Consider now the right hand side in \eqref{e:Nuf}.
Using $\ZZ_j =L_Z  \Q_j$, the distribution of $L_Z$ calculated in \eqref{e:L_Z} and independence shown in Proposition \ref{P:main_pomocni}, we have
\begin{eqnarray*}
\lefteqn{
\exp \left[-\int_0^1\left(1-\EE e^{-\sum_j f(t,u\ZZ_j)}\right) \theta u^{-\alpha} dt\right]}\\
&=&\exp \left[-\int_0^1\int_1^{\infty}\left(1-\EE e^{-\sum_j f(t,u l \Q_j)}\right) \alpha l^{-\alpha-1}dl\theta u^{-\alpha} dt\right]\\
&=&\exp \left[-\int_0^1\int_u^{\infty}\left(1-\EE e^{-\sum_j f(t,v \Q_j)}\right) \alpha v^{-\alpha-1}dv\theta  dt\right]\\
&=&\exp \left[-\int_0^1\int_0^{\infty}\left(1-\EE e^{-\sum_j f(t,v\Q_j)}\right) \theta \alpha v^{-\alpha-1}dv dt\right]=\EE e^{-N(f)}, 
\end{eqnarray*}
where the second equality follows using the change of variable $v=ul$, and the last equality follows from the fact that $\sup_j \|\Q_j\|=1$ and $f(t,x)=0$ for $x<u$. 
\end{proof}

As we discussed in the introduction, one consequence of the theorem above is the functional limit theorem
for partial sums of the univariate sequence $(X_t)$, see Theorem~3.4 in \cite{BKS}.
In the multivariate case, under the conditions of Theorem~\ref{main},  it was shown 
by Davis and Mikosch~\cite{DM98} that partial sums $S_n = \X_1+ \cdots + \X_n,\ n\geq 1\,,$ satisfy
\begin{equation} \label{LimThm}
\frac{S_n}{a_n} \dto \xi_\alpha\,,
\end{equation}
for $\alpha\in (0,1)$ and   an $\alpha$-stable random vector $\xi_\alpha$,
actually the same holds for $\alpha\in [1,2)$ if one assumes for instance  that $\X_i$'s have a symmetric distribution and that  for any $\delta>0$ the following standard
technical assumption holds
$ \lim_{\vep\to 0} \limsup_{n\toi} \P (\| a_n^{-1} \sum_{i=1}^n \X_i \1_{\{\| \X_i \| \leq \vep a_n  \}} \| > \delta)=0$\,, see~\cite{DM98} again.

 As observed by Mikosch and Wintenberger~\cite{MW13}, under condition that 
 \[
  \EE \left( \dsum_{j\geq 1} \| \Q_j \| \right)^{\alpha} < \infty\,,
 \]
which indeed always holds in the case  $\alpha<1$ or if $(\X_n)$ is $m$--dependent,
the spectral measure $\Gamma_\alpha$ of the stable random vector $\xi_\alpha$  in \eqref{LimThm}
can be characterized as follows 
\begin{align*}
\int_{\mathbb{S}^{d-1}} \langle \tB ,\sB  \rangle^\alpha_+ \Gamma_\alpha(d\sB) 
& = \theta \frac{\alpha}{2-\alpha} \EE \left[ \left( \sum_{j\geq 1} \langle \tB ,\Q_j \rangle  \right)^\alpha_+  \right] \\
&=  \theta \frac{\alpha}{2-\alpha} \EE \left[ \left( \sum_{j\geq 1}
 \langle \tB ,\Y_j /\sup_j \|\Y_j \| \rangle   \right)^\alpha_+  
 \,\big| \, \sup_{i \le -1} \|\Y_i\|\leq 1 
 \right] \,,
\end{align*}
where $\tB \in \mathbb{S}^{d-1}$, $\langle \tB , \sB  \rangle$ denotes the inner product on 
on $\mathbb{S}^{d-1}$ and $\| \cdot \|$ denotes the Euclidean norm. 
For a definition of a the spectral measure of a stable random vector, see \cite[Section 2.3]{ST}. 
This can be used in certain cases
to give   an alternative representation of the so-called cluster index $ b(\tB) \in \mathbb{S}^{d-1}$ studied by
Mikosch and Wintenberger in \cite{MW13,MW15}.  Indeed, for a regularly varying  $(\X_n)$ which is  Markov chain  under conditions of \cite[Theorem 4.1]{MW13}
\[
 b(\tB) = \frac{ 1-\alpha}{ \Gamma (2-\alpha) \cos (\pi \alpha /2)}
 \int_{\mathbb{S}^{d-1}} \langle \tB ,\sB  \rangle^\alpha_+ \Gamma_\alpha(d\sB) \,.
\]
As observed in \cite{MW13,MW15}, the cluster index $b(\tB)$ plays the key role in the 
description of  large deviations for partial sums of regularly varying sequences.

A more direct consequence of Theorem~\ref{main} is  the functional limit theorem
for the partial maxima of a univariate sequence considered in the following section.

\section{Invariance principle for maximal process} \label{Sec:etrem}

Denote by  $M'_n = \max\{X_1,\ldots , X_n\},\ n\geq 1$.
Following Resnick~\cite{Res87} or Embrechts et al. \cite{EKM} we consider the following
continuous time  c\`adl\`ag  process
\[
 Y_n (t) = \left\{ \begin{array}{cr}
  M'_{\lfloor nt \rfloor}/ a_n & t\geq 1/n\,,\\
  X_1 /a_n & t <1/n\,,\\
 \end{array} \right.
\]
indexed over the segment $[0,1]$. It is actually customary to exclude $t=0$ 
to avoid technical problems, see \cite{Res87}. In our approach, such issues are easily avoided, and
therefore we include $t=0$. From the practical perspective, it seems sufficient to  consider  $t\leq 1$. Extending the main result below to  $t \in [0,\infty)$ remains a  technical, but relatively straightforward task, see Chapter 3 in Billingsley \cite{Bil2}.

Recall that the extremal process generated by 
an extreme value distribution function $G$ ($G$-extremal process, for short) is a continuous time stochastic process 
with finite dimensional distributions $G_{s_1,\ldots,s_k}$ 
satisfying
\[ 
 G_{s_1,\ldots,s_k}(x_1,\ldots,x_k)=G^{s_1}(\wedge_{i=1}^k x_i)G^{s_2-s_1}(\wedge_{i=2}^k x_i)\cdots G^{s_k-s_{k-1}}(x_k)  \,,
 \]
 for all choices of $k\geq 1$, $0<s_1<\cdots < s_k$, $x_i\in\mathbb{R}$, $i=1,\ldots,k$, see Resnick~\cite{Res87}.
 The processes $Y_n$ introduced above are clearly random elements in  the space of real valued c\`adl\`ag functions  $D[0,1]$. We will show that they converge weakly to a particular extremal process
with respect to  Skorohod's $M_1$ topology on $D[0,1]$. We refer to  Whitt~\cite{Whitt02} for the 
definition and discussion of various topologies in that space, see also~\cite{BKS}. 
 
In the proposition below we make a small technical  assumption about the right tail of the marginal distribution of $X_t$'s, i.e. we suppose that it satisfies
 $ \liminf_{n\toi} n \P(X_0 > a_n) >0$. This implies that 
\begin{equation}\label{e:a_np}
n\P(X_0 > a_n x)\to p x^{-\alpha}\quad \textrm{ for all } x>0,
\end{equation}
and some constant $p \in (0,1]$.

\begin{Prop}\label{Thm:ExProc}
Let $(X_n)_{n\in \Z}$ be a stationary sequence of jointly regularly varying random variables with index $\alpha$, satisfying conditions $\mathcal{A}'$ and $\mathcal{AC}$. Assume that $ \liminf_{n\toi} n \P(X_0 > a_n) >0$,
then\\[2mm]
\begin{equation} \label{main_2}
Y_n  \Rightarrow \xi,
\end{equation}
where $\xi(t),\ t>0$ is a $G$-extremal process for  the nonstandard 
Fr\'echet distribution function
$
 G(x) = e^{-\kappa x^{-\alpha}},\  x\geq 0\,,
$
for some $\kappa>0$,
and the convergence takes place in $D[0,1]$ endowed with the Skorohod's $M_1$ topology.
\end{Prop}
\begin{proof}
%
Consider the 
functional  $T^+:M_p([0,1]\times \E) \to D[0,1]$ on the space of Radon point measures   given by
\begin{equation}\label{def_functional_Tplus}
T^+(m)(t)= \sup_{t_i\leq t} j_i \vee 0  \,,
\end{equation}
where $m=\sum_{i}\delta_{(t_i,j_i)}$ denotes an arbitrary Radon point measure on the space $[0,1]\times \mathbb{E}$, where we set for convenience
$\sup \emptyset = 0$.

Denote  $M'_p =\{m \in M_p: m([0,s]\times (0,\infty])>0  \mbox{ for all } s>0 \}$ and 
$M''_p =\{m \in M_p: m(\{0,1\}\times (0,\infty])=0  \}$. We will show that
$T^+$ is continuous on the set $M'_p \cap M''_p$. Assume, $m_n \vto m \in M'_p \cap M''_p$.
Because, $m$ is a Radon point measure, 
the set of times $t$ for which $m(\{t\}\times \E ) = 0$ is dense
in $[0,1]$. For all such $t$'s
$$
 T^+ m_n (t) \to T^+ m (t)\,,
$$ 
as $n\toi$ by Proposition 3.13 in \cite{Res87}. Observe that  $T^+ m$ is a nondecreasing function
for any $m$, therefore an application of Corollary 12.5.1 in Whitt~\cite{Whitt02} yields
the convergence of $ T^+ m_n  \to T^+ m $  in $D[0,1]$ endowed with  $M_1$ topology.

Observe now that the limiting point process $N$ in Theorem~\ref{main} lies in  $ M'_p \cap M''_p$
a.s.  Hence for $N_n$ in the same theorem, an application of the continuous mapping theorem yields the weak convergence in
$$
T^+ N_n \Rightarrow  T^+ N \,,
$$
as $n\toi$.

By Slutsky argument, to show that $Y_n \Rightarrow  T^+ N$, it is sufficient to prove that 
$$
 d_{M_1} (Y_n,T^+ N_n) \pto 0\,,
$$ %
where $d_{M_1}$ denotes the $M_1$ metric on the space $D[0,1]$, see Whitt~\cite{Whitt02}.
This follows from the fact that
 metric $d_{M_1}$ is  weaker than the uniform metric $d_\infty$, and the following obvious limit
$$
 d_{\infty} (Y_n,T^+ N_n)  = \frac{|X_1|}{a_n} \1_{\{ X_1<0 \}} \asto 0\,,
$$
as $n\toi$. 

Observe that, by definition,  $N( [0,t]\times(0,\infty] ) =0$  implies  $T^+ N(t) =0$ for a fixed  $t\geq 0$. 
Therefore
\begin{equation}\label{LimExtProc}
T^+ N (t)  = \sup_{T_i\leq t} {P_i \sup \eta_{ij}} =  \sup_{T_i\leq t} {P_i U_i} \,,
\end{equation}
denoting  $U_i =  \sup_j \eta_{ij} \vee 0 $. Note further that $U_i'$s form an iid sequence of random variables  on the interval $[0,1]$. By \eqref{e:a_np}, $\P(U_i \in (0,1]) >0$.
Moreover, sequence $(U_i)$ is independent of the PRM $\sum_i \delta_{T_i,P_i}$.  It is straightforward to
check that $\sum_i \delta_{T_i,P_i U_i}$ is PRM  with  mean measure 
 $Leb \times \nu'$ where $\nu'(dy)= \theta \EE U^\alpha \alpha y^{-\alpha-1}dy$ for $y>0$,
 cf. propositions 3.7 and 3.8 in Resnick~\cite{Res87}. 
 It follows, by proposition 5.4.4 in Embrechts et al.~\cite{EKM} for instance, that 
 $T^+ N (t),\ t>0$ in \eqref{LimExtProc} is a $G$--extremal process for 
 $$
 G(u) =  \P \left(  T^+ N (1) \leq u \right) = 
  \P \left(   \sum_i \delta_{P_i U_i}  (u,\infty]  = 0  \right) 
 = \exp \left(  - \theta \EE U^\alpha u ^{-\alpha} \right)\,,
 $$
for any $ u >0.$


\end{proof}
 
By the proof,  the constant $\kappa$  in Proposition~\ref{Thm:ExProc} equals
$$
 \kappa = \theta \,  \EE U^\alpha\,,
$$
 where $U \eqd  \sup_j \eta_{ij} \vee 0 $. In the iid case, the extremal index $\theta =1$,
 while $ \sum_j \delta_{\eta_{ij}} \eqd \delta_Q$ where $\P(Q=1) = 1- P(Q=-1) = p$.
 Therefore, $\kappa= p$ as known from Proposition 4.20 in Resnick~\cite{Res87}.
 In the case of nonnegative random variables $p=1$ and $\sup \eta_{ij} = 1$ a.s. Therefore
 $\kappa = \theta $ in this case, cf.  Krizmani\'c~\cite{DK14}.
  
\begin{Rem} 
Observe that more commonly used $J_1$ topology is not applicable in our setting, because of the clustering of extremes.
For an illustration of the problem, consider the  point measures $m_n= \delta_{1/2-1/n,1/2} + \delta_{1/2,1}
\vto m =\delta_{1/2,1/2} + \delta_{1/2,1} $ for $n\toi$,  and note that for $T^+$ in \eqref{def_functional_Tplus},
$ T^+ m_n  $ does not converge to $ T^+ m $  in $J_1$ topology.

In the case  $\theta=1$, the limiting point process $N$ is a Poisson random measure, and the convergence in the theorem above holds in the standard $J_1$ topology. The proof only has to be adapted
to show that $T^+$ is an a.s. continuous functional with respect to
the distribution of such a process. Such a result was already stated in Remark 2 of Mori~\cite{Mo77}.
\end{Rem}

\noindent{\bf Acknowledgements}: This work has been supported in part by Croatian Science Foundation under the projects 1356 and 3526.

\bibliographystyle{spmpsci}      
  
\end{document}